\documentclass{amsart}
\usepackage{amssymb, amsmath, mathrsfs, verbatim, url, bbm}

\theoremstyle{plain}
\newtheorem{theorem}{Theorem}[section]
\newtheorem{lemma}[theorem]{Lemma}

\newtheorem{corollary}[theorem]{Corollary}

\newcommand{\MainTheoremName}{Main Theorem}

\theoremstyle{definition}
\newtheorem{definition}[theorem]{Definition}
\newtheorem{example}[theorem]{Example}

\theoremstyle{remark}
\newtheorem{remark}[theorem]{Remark}

\newtheorem{question}[theorem]{Question}

\numberwithin{equation}{section}

\DeclareMathOperator{\cf}{cf}
\DeclareMathOperator{\ci}{ci}

\newcommand{\nbd}{\nobreakdash}

\newcommand{\card}[1]{\left\lvert #1\right\rvert}

\newcommand{\la}{\langle}
\newcommand{\ra}{\rangle}

\newcommand{\cardd}{\mathfrak{d}}
\newcommand{\cardc}{\mathfrak{c}}
\newcommand{\cardp}{\mathfrak{p}}
\newcommand{\cardu}{\mathfrak{u}}
\newcommand{\inv}[1]{#1^{-1}}
\newcommand{\closure}[1]{\overline{#1}}
\newcommand{\restrict}{\upharpoonright}

\newcommand{\cell}[1]{c\left(#1\right)}
\newcommand{\hatcell}[1]{\hat{c}\left(#1\right)}

\newcommand{\weight}[1]{w(#1)}

\newcommand{\character}[1]{\chi(#1)}
\newcommand{\minchar}[1]{\mathrm{min}\chi(#1)}
\newcommand{\hattight}[1]{\hat{\mathrm{t}}(#1)}
\newcommand{\hatfree}[1]{\hat{\mathrm{F}}(#1)}
\newcommand{\free}[1]{\mathrm{F}(#1)}
\newcommand{\tight}[1]{\mathrm{t}(#1)}
\newcommand{\pseudocharacter}[1]{\psi(#1)}

\newcommand{\reals}{\mathbb{R}}

\newcommand{\wma}{we may assume}

\newcommand{\iso}{\cong}

\newcommand{\tlex}{\mathrm{lex}}

\newcommand{\mcA}{\mathcal{ A}}
\newcommand{\mcB}{\mathcal{ B}}
\newcommand{\mcC}{\mathcal{ C}}
\newcommand{\mcD}{\mathcal{ D}}

\newcommand{\mcF}{\mathcal{ F}}
\newcommand{\mcG}{\mathcal{ G}}

\newcommand{\mcU}{\mathcal{ U}}
\newcommand{\mcV}{\mathcal{ V}}

\newcommand{\mbP}{\mathbb{ P}}

\newcommand{\om}{\omega}
\newcommand{\oml}{\omega_1}
\newcommand{\betaoo}{\beta\om\setminus\om}
\newcommand{\al}{\aleph}
\newcommand{\ka}{\kappa}
\newcommand{\lm}{\lambda}
\newcommand{\ie}{{\it i.e.},}

\newcommand{\arhan}{Arhan\-gel${}^\prime$ski\u\i}
\newcommand{\shap}{Shapirovski\u\i}

\newcommand{\locsplit}[3][.]{\mathrm{split}_{#2}^{\ifx#1.{}\else{#1}\fi}(#3)}
\newcommand{\cardm}{\mathfrak{m}}
\newcommand{\cardml}[1]{\mathfrak{m}_{\sigma\mathrm{-}#1\mathrm{-linked}}}
\newcommand{\cardmcen}{\mathfrak{m}_{\sigma\mathrm{-centered}}}
\newcommand{\cardmctbl}{\mathfrak{m}_{\mathrm{countable}}}
\newcommand{\todor}{Todor\v cevi\'c}

\newcommand{\trileq}{\trianglelefteq}

\newcommand{\nbhd}[1]{\mathrm{Nbhd}\left(#1\right)}
\newcommand{\eqcf}{\equiv_\mathrm{cf}}
\newcommand{\Reg}{\mathrm{Reg}}
\newcommand{\escape}[1]{\mathrm{Escape}\left(#1\right)}
\newcommand{\escapero}[1]{\mathrm{Escape}_{\mathrm{RO}}\left(#1\right)}
\newcommand{\add}[1]{\mathrm{Add}\left(#1\right)}
\newcommand{\ro}[1]{\mathrm{RO}\left(#1\right)}

\title{Forbidden rectangles in compacta}

\author{
David Milovich}
\address{
Department of Engineering, Mathematics and Physics \\
Texas A\&M International University \\
5201 University Blvd\\
Laredo, TX\\
78041 USA}

\email{david.milovich@tamiu.edu}



\begin{document}

\begin{abstract}
We establish negative results about ``rectangular'' 
local bases in compacta. For example, there is no compactum 
where all points have local bases of cofinal type
$\omega\times\om_2$. For another, the compactum $\beta\om$ has 
no nontrivially rectangular local bases, and the same is consistently 
true of $\betaoo$: no local base in $\beta\om$ has cofinal type 
$\kappa\times\cardc$ if $\kappa<\cardml{n}$ for some $n\in[1,\om)$.
Also, CH implies that every local base in $\betaoo$ has
the same cofinal type as one in $\beta\om$.

We also answer a question of Dobrinen and \todor\ about cofinal 
types of ultrafilters: the Fubini square of a filter on $\om$ 
always has the same cofinal type as its Fubini cube. Moreover,
the Fubini product of nonprincipal P-filters on $\om$ is commutative
modulo cofinal equivalence.
\end{abstract}

\maketitle

\section{Introduction}\label{intro}

Recall that a space $X$ is \emph{homogeneous} if for all $p,q\in X$,
$h(p)=q$ for some autohomeomorphism of $X$.
Many questions about compact homogeneous spaces 
are still unsolved (in all models of ZFC) after decades;
see~\cite{vanmillopit2} for a survey of these questions.
For example, Rudin's Problem asks whether every homogeneous compactum
has a convergent sequence.
(A \emph{compactum} is a compact Hausdorff space.)
Our motivating question is Van Douwen's Problem: 
Is there a homogeneous compactum $X$ with 
a pairwise disjoint family $\mcF$ of open sets
such that $\card{\mcF}>\card{\reals}$.
See~\cite{kunenlhc} for more about Van Douwen's Problem. 

Pairwise disjoint families of open sets are called \emph{cellular families}
for short;
the \emph{cellularity} $\cell{X}$ of a space $X$ is the supremum of
the cardinalities of its cellular families. Oversimplifying, 
Van Douwen's Problem is hard because if we seek an infinitary
operation on spaces that preserves both homogeneity and compactness, 
we apparently find only the product operation and a special quotient
operation~\cite{milovicha}. However,
the cellularity of a product is just the supremum of 
the cellularities of its finite subproducts, and taking a
quotient never increases cellularity.

What might a homogeneous compactum with large cellularity ``look like''?
One way to make this question more precise is 
to ask for examples of directed sets $(D,\leq)$ such that 
any homogeneous compactum $X$ with a local base $\mcB$ satisfying
$(\mcB,\supseteq)\iso(D,\leq)$ will satisfy some lower bound
of $\cell{X}$.  Since every infinite compactum has a countable
set with a limit point, homogeneity implies that any such $D$
must have a countable unbounded set. However, for any directed
$D$, $\omega\times D$ has a countable unbounded set. So, for a
simple, ``rectangular'' example, if any space $X$ 
has a clopen local base $\mcB$ of order type $\omega\times\kappa$ 
(where $\kappa$ is an infinite cardinal), then $\cell{X}\geq\kappa$
because if $U\colon \omega\times\kappa\iso\mcB$, then 
$\{U(0,i)\setminus U(0,i+1):i<\kappa\}$ is a cellular family.
In Section~\ref{rectcell}, we will remove the assumption 
of clopenness in exchange for assuming merely that $X$ is $T_3$.
Moreover, we will relax ``order type'' to ``cofinal type.''

Two local bases at a common point $p$ in a space $X$ 
could have different order types, but with respect to
the containment ordering $\supseteq$, they are both 
cofinal subsets of the neighborhood filter $\nbhd{p,X}$ 
(our notation for the set of all $U\subseteq X$ 
where $p$ is in the interior of $U$). 
Therefore, it is more natural to investigate
a local base's cofinal type than its order type, where 
preorders $P$ and $Q$ are \emph{cofinally equivalent} 
if there is a preorder $R$ such that $P$ and $Q$
are order isomorphic to cofinal subsets of $R$.
(It is not hard to check that this is an equivalence relation.)
Since all local bases at $p$ are cofinally equivalent 
to $\nbhd{p,X}$, we will state our results
in terms of the cofinal type of $\nbhd{p,X}$, with
the convention that families of subsets of a space
are ordered by containment.

In section~\ref{rectcell}, we establish that if a $T_3$
space $X$ has a neighborhood filter cofinally equivalent to
a finite product $\prod_{i\leq n}\ka_i$ of regular cardinals
$\ka_0<\cdots<\ka_n$, then $\cell{X}\geq\ka_n$.
In section~\ref{skinny}, we prove that not all points in
a compactum can have a fixed ``skinny'' cofinal type.
For example, given $\ka_0<\cdots<\ka_n$ as above, if
$\ka_m^+<\ka_{m+1}$ for some $m<n$, then not all neighborhood
filters in a compactum can be cofinally equivalent to 
$\prod_{i\leq n}\ka_i$.  So, in a homogeneous compactum,
no neighborhood filter is cofinally equivalent to such
$\prod_{i\leq n}\ka_i$. Section~\ref{skinny} then goes on
to a stronger theorem for homogeneous compacta in models
of GCH. A corollary is that the supremum of the 
cardinalities of the free sequences in a homogeneous 
compactum $X$ is always attained if GCH holds.
Left open is whether a homogeneous compactum 
could have a neighborhood filter cofinally equivalent 
to $\prod_{i\leq n}\om_i$ for some $n\geq 1$. 

In section~\ref{ultrafilters}, we shift our attention
to the inhomogeneous compactum $\betaoo$. In this space,
every point corresponds to a nonprincipal ultrafilter on $\om$, 
and the cofinal type of a neighborhood filter (ordered by
containment) is the same as the cofinal type of the
corresponding ultrafilter ordered by eventual containment.
Among other things, we observe that $\betaoo$ consistently has 
no neighborhood filter cofinally equivalent to any
$\prod_{i\leq n}\ka_i$ as above. In ZFC, it is known that
no neighborhood filter of $\betaoo$ is cofinally equivalent
to $\om\times\cardc$~\cite{milovichtuk}. We extend this
result, ruling out $\ka\times\cardc$ for all 
$\ka<\sup_{n<\om}\cardml{n}$, 
which is at worst very close to optimal because 
it is consistent to have $\cardmcen<\cardc$ and
a neighborhood filter of $\betaoo$ cofinally equivalent
to $\cardmcen\times\cardc$.

In section~\ref{fubini}, we prove some 
results about cofinal types of neighborhood bases in
$\beta\om$, which are exactly the cofinal types of 
neighborhood bases of ultrafilters on $\om$. In particular,
we answer a question of Dobrinen and \todor~\cite{dobtor} 
by showing that the Fubini square and Fubini cube of
a filter on $\omega$ are always cofinally equivalent.
(The ordering is containment.) We also establish
commutativity modulo cofinal equivalence for the
Fubini product of nonprincipal $P$-filters on $\om$.
These results follow from our ``rectangular''
characterization of the Fubini product: if $F$ and $G$
are nonprincipal filters on $\omega$, then the Fubini product of
$F$ and $G$ is cofinally equivalent to $F\times G^\omega$.

\section{Rectangles and cellularity}\label{rectcell}

Given an ordinal $\alpha$, let $2^\alpha_\tlex$ be 
${}^{\alpha}2$ with the topology induced by the lexicographic
ordering. If $\ka_0<\cdots<\ka_n$ are regular cardinals
and $\lm_i$ is a regular cardinal $\leq\ka_i$, for all $i$, 
then it easy to find a point $p$ in 
$\prod_{i\leq n}2^{\ka_i}_\tlex$ such that each
$p(i)$ has cofinality $\lm_i$ and coinitiality $\ka_i$ in
$2^{\ka_i}_\tlex$, where
our convention is that $1$ is the unique finite regular cardinal,
every minimum of a preorder has cofinality $1$, and 
every maximum of a preorder has coinitiality $1$.
Conversely, for every $p\in\prod_{i\leq n}2^{\ka_i}_\tlex$ and
$i\leq n$, one of $\cf(p(i))$ and $\ci(p(i))$ is $\ka_i$,
and the other is some regular $\lm_i\leq\ka_i$. Therefore,
the spectrum of cofinal types of $\prod_{i\leq n}2^{\ka_i}_\tlex$
is 
$
\left\{\prod_{i\leq n}(\lm_i\times\ka_i)/\eqcf:\ka_i\geq\lm_i\in\Reg\right\}
$
where $\Reg$ is the class of regular cardinals.

Since the diagonal of $\ka_i\times\ka_i$ is cofinal,
$\ka_i\times\ka_i\eqcf\ka_i\eqcf 1\times\ka_i$.
Therefore, all neighborhood filters of 
$\prod_{i\leq n}2^{\om_i}_\tlex$ are cofinally equivalent to
$\prod_{i\leq n}\om_i$. However, this space, though compact
Hausdorff, is not homogeneous if $n\geq 1$. 
To see this, recall that the 
\emph{$\pi$\nbd-character} $\pi\character{p,X}$ 
of a point $p$ in a space $X$ is the minimum of 
the cardinalities of the \emph{local $\pi$-bases} at $p$,
where a local $\pi$-base at $p$ is a family of nonempty open sets
that includes a subset of every neighborhood of $p$.
If some point $p$ in a linearly ordered space $X$ has
cofinality $\lm\geq\om$, then $p$ has a local $\pi$-base
of size $\lm$: $\{(q_i,q_{i+1}):i<\lm\}$ for some
increasing sequence $\vec{q}$ converging to $p$.
Moreover,
$\pi\character{p,X}=\min(\{\cf(p),\ci(p)\}\setminus\{1\})$.
Therefore, if $p\in\prod_{i\leq n}2^{\om_i}_\tlex$, then
$$
\pi\character{p,X}=\max_{i\leq n}\,
\min(\{\cf(p(i)),\ci(p(i))\}\setminus\{1\}).
$$
It follows that $\prod_{i\leq n}2^{\om_i}_\tlex$ has points
with $\pi$\nbd-character $\om_i$, for all $i\leq n$.
Thus, $\prod_{i\leq n}2^{\om_i}_\tlex$ is only homogeneous
in the trivial case $n=0$.
More generally, it is shown in~\cite{arhanprodlots} that 
for any homogeneous compact product of linear orders, all
factors $X$ are such that every $p\in X$ satisfies
$\{\cf(p,X),\ci(p,X)\}\subseteq\{1,\om\}$.

\begin{question}\label{homogrect}
Is there a homogeneous compactum with a neighborhood
filter cofinally equivalent to $\prod_{i\leq n}\om_i$ for
some $n\geq 1$?
\end{question}

There is some weak evidence in~\cite{milovichnth} for a ``no''
answer to the above question. 
Suppose that $Y$ is a homogeneous compactum with a neighborhood
filter cofinally equivalent to $\prod_{i\leq n}\om_i$ for
some $n\geq 1$. If $Y$ also had a point of uncountable $\pi$\nbd-character, 
then, by Theorem 5.7 of~\cite{milovichnth}, there would be
a Tukey map (see Definition~\ref{deftukey}) from $[\om_1]^{<\om}$ 
(ordered by $\subseteq$) to $\prod_{i\leq n}\om_i$.
However, it is well known that there is no such Tukey map.
(For a quick proof, check that
every uncountable subset of $\prod_{i\leq n}\om_i$ has an
infinite bounded subset, and that this property is precisely
the negation of having a Tukey map from $[\om_1]^{<\om}$.)
Thus, a positive answer to Question~\ref{homogrect} requires
a homogeneous compactum whose points all have 
$\pi$\nbd-character $\om$ but character $\om_n$.
(See Definition~\ref{defchar}.)
In almost all known homogeneous compacta $X$, $\pi$\nbd-character
equals character at all points.
The only known class of exceptions was discovered by Van Mill
in~\cite{vanmillindep}, and even these exceptions consistently 
do not exist: they are homogeneous if $\mathrm{MA}+\neg\mathrm{CH}$ 
holds but inhomogeneous if CH holds.
Moreover, it is shown in~\cite{milovichnth} that in
all known homogeneous compacta $X$, all points $p$ are \emph{flat},
that is, satisfy $\nbhd{p,X}\eqcf[\ka]^{<\om}$
where $\ka$ is the character of $p$. 
(In particular, Van Mill's exceptional homogeneous compacta
are all separable and have weight less than $\cardp$;
by Theorem 2.16 of~\cite{milovichnth}, any homogeneous
compactum satisfying these two properties has only flat points.)

Question~\ref{homogrect} is relevant to Van Douwen's Problem
because of the next theorem.

\begin{theorem}\label{t3rectcell}
If $\ka_0<\cdots<\ka_n$ are regular cardinals, $X$ is a $T_3$ space, 
$p\in X$, and $\nbhd{p,X}\eqcf\prod_{i\leq n}\ka_i$,
then $\cell{X}\geq\ka_n$.
\end{theorem}

We will actually prove a stronger result
and deduce the above theorem as a corollary.

\begin{definition}
Given a point $p$ in a space $X$,
\begin{itemize}
\item an \emph{escape sequence} at $p$ is a transfinite sequence of
neighborhoods $U_0\supseteq U_1\supseteq\cdots\supseteq U_i\supseteq\cdots$
of $p$ such that $\bigcap_i U_i$ is not a neighborhood of $p$;
\item $\escape{p,X}$ is the set of infinite cardinals $\ka$
for which $p$ has a $\ka$\nbd-long escape sequence;
\item $\escapero{p,X}$ is the set of infinite cardinals $\ka$
for which $p$ has a $\ka$\nbd-long escape sequence 
consisting of regular open sets;
\item $\hatcell{X}$ is the least cardinal $\ka$ such that $X$
lacks a cellular family of size $\ka$. 
(For increased precision at limit cardinals, 
we use $\hatcell{X}$ instead of $\cell{X}$.)
\end{itemize}
\end{definition}

\begin{theorem}\label{escapecell}
If $X$ is a space and $\ka\in\escapero{p,X}\cap\Reg$ for some $p\in X$,
then $\hatcell{X}>\ka$.
\end{theorem}
\begin{proof}
Let $\vec{U}$ be a $\ka$\nbd-long regular open escape sequence at $p$.
Since $\ka$ is regular and $\vec{U}$ cannot be eventually constant,
we may thin out $\vec{U}$ such that it is strictly decreasing.
Since each $U_i$ and $U_{i+1}$ are regular open, 
each $U_i\setminus\closure{U}_{i+1}$ is nonempty, 
so $\{U_i\setminus\closure{U}_{i+1}:i<\ka\}$ is a cellular family
of size $\ka$.
\end{proof}

\begin{example}
If $X=2^{\om_1}$ (with the product topology), then there is an
$\om_1$\nbd-long escape sequence $\vec{U}$ at $\vec{0}$
given by $U_i=\bigcup_{i\leq j<\oml}\inv{\pi_j}[\{0\}]$.
However, $\hatcell{X}=\om_1$ (by a well-known $\Delta$-system argument).
So, by Theorem~\ref{escapecell}, at no point in $X$ 
is there a regular open escape sequence with length $\om_1$.
Thus, despite the set of regular open neighborhoods of $\vec{0}$ being
cofinally equivalent with the set of open neighborhoods of $\vec{0}$,
the former never has unbounded increasing $\om_1$-sequences, while the
latter does. Moreover,
the neighborhood filter of $\vec{0}$ in $X$ is cofinally
equivalent to the neighborhood filter of $\infty$ in the one\nbd-point
compactification of the $\oml$\nbd-sized discrete space, yet only the
former point has an $\om_1$\nbd-long escape sequence.
\end{example}

\begin{definition} Given a preorder $P$,
\begin{itemize}
\item $P$ is \emph{$\ka$-directed} if every $A\in[P]^{<\ka}$ 
has an upper bound in $P$;
\item $P$ is \emph{directed} if it is $\om$\nbd-directed;
\item the \emph{additivity} $\add{P}$ of $P$ is the least cardinal $\lm$ 
for which $P$ is not $\lm$-directed, if it exists;
\item if $P$ has a maximum, then $\add{P}=\infty$.
\end{itemize}
\end{definition}

\begin{remark}\
\begin{itemize}
\item $\add{P}$ is always $1$, $2$, $\infty$, or a regular infinite cardinal.
\item $P$ is cofinally equivalent to a regular infinite cardinal $\mu$
if and only if $\add{P}=\cf(P)=\mu$.
\item If $P\eqcf Q$, then $\add{P}=\add{Q}$ and $\cf(P)=\cf(Q)$. 
\end{itemize}
\end{remark}

\begin{lemma}\label{addpartition}
If $A$ is a preorder, $f\colon A\rightarrow B$, and $\card{B}<\add{A}$,
then $f$ is constant on a cofinal subset of $A$.
\end{lemma}
\begin{proof}
Let $N$ denote the set of all $b\in B$ for which
the fiber $\inv{f}\{b\}$ is not cofinal in $A$.
For each $b\in N$, choose $a(b)\in A$ not bounded above 
by anything in $\inv{f}\{b\}$. 
Since $\card{B}<\add{A}$, there is an upper bound $c$
of $\{a(b):b\in N\}$. 
Since $f(c)$ cannot be in $N$, we have $N\not= B$.
\end{proof}

\begin{lemma}\label{escapegap}
If $\ka$ is a regular infinite cardinal, $D$ and $E$ are preorders,
$X$ is a space, $p\in X$, $\nbhd{p,X}\eqcf D\times E$, 
and $\cf(D)<\ka<\add{E}$, then $\ka\not\in\escape{p,X}$.
\end{lemma}
\begin{proof}
We may assume $\nbhd{p,X}$ and $D\times E$ are disjoint,
so there is a preordering $\trileq$ of 
$\nbhd{p,X}\cup(D\times E)$ that makes
$\nbhd{p,X}$ and $D\times E$ cofinal suborders.
Suppose that $\vec{U}$ is a $\ka$\nbd-long sequence in $\nbhd{p,X}$.
It suffices to show that $\vec{U}\restrict J$ is bounded for
some cofinal $J\subseteq\ka$.
Let $(d_i:i<\cf(D))$ enumerate a cofinal subset of $D$.
Every $j<\ka$ is such that $U_j\trileq (d_{i(j)},e_j)$ 
for some $i(j)<\cf(D)$ and $e_j\in E$. 
By Lemma~\ref{addpartition}, since $\cf(D)<\add{\ka}$, 
there must be some $l<\cf(D)$ such that
$i(j)=l$ for all $j$ in a cofinal subset $J$ of $\ka$.
Since $\ka<\add{E}$,  there exists $e\in E$ such that 
$U_j\trileq (d_l,e)$ for all $j\in J$.
\end{proof}

\begin{remark}\label{cfaddgap}
The above proof works if we replace $\nbhd{p,X}$ 
with an arbitrary preorder. In particular, 
if $\ka$ is regular infinite cardinal, $D$ and $E$ are nonempty preorders,
and $\cf(D)<\ka<\add{E}$, then $D\times\ka\times E\not\eqcf D\times E$
because $D\times\ka\times E$ has an unbounded increasing sequence 
of length $\ka$.
\end{remark}

\begin{lemma}\label{escaperofilled}
If $X$ is $T_3$, $p\in X$, $D$ and $E$ are preorders,
$\nbhd{p,X}\eqcf D\times E$, and $\cf(D)<\add{E}<\infty$,
then $\add{E}\in\escapero{p,X}$.
\end{lemma}
\begin{proof}
We may assume $\nbhd{p,X}$ and $D\times E$ are disjoint,
so there is a preordering $\trileq$ of $\nbhd{p,X}\cup(D\times E)$ 
that makes $\nbhd{p,X}$ and $D\times E$ cofinal suborders.
Since $X$ is $T_3$, the regular open neighborhoods of $p$ also
form a cofinal suborder---call it $\ro{p,X}$. 
For each $(d,e)\in D\times E$, let $U(d,e)$ be the smallest 
(\ie\ $\trileq$\nbd-greatest) regular open neighborhood of $p$ 
that contains (\ie\ is $\trileq$\nbd-below) every $V\in\ro{p,X}$
satisfying $(d,e)\trileq V$. 
For every $(d,e)\in D\times E$, there exist $V\in\ro{p,X}$ 
and $(d',e')\in D\times E$ such that $(d,e)\trileq V\trileq (d',e')$, 
which implies $(d,e)\trileq U(d',e')$.
So, choose $(f,g)\colon D\times E\rightarrow D\times E$ such that
$(d,e)\trileq U((f,g)(d,e))$ for all $(d,e)\in D\times E$.

By replacing $D$ with a cofinal subset if necessary,
\wma\ that $\card{D}=\cf(D)$. Fix $d\in D$.
By Lemma~\ref{addpartition}, since $\card{D}<\add{E}$, 
there is a cofinal subset $E_0$ of $E$ 
such that $f(d,e)=f(d,e')$ for all $e,e'\in E_0$; 
let $b=f(d,e)$ for any $e\in E_0$.
For each $e\in E$, let $G(e)=g(d,e')$ for some
$e'\in E_0$ where $e\leq e'$; we then have
$(d,e)\trileq U(b,G(e))$ for all $e\in E_0$.
Set $\ka=\add{E}$; choose $A=\{a_i:i<\ka\}\subseteq E$ such
that $A$ is unbounded in $E$. Choose $(e_i:i<\ka)$ such that 
$\{a_i,e_j,G(e_j)\}\leq e_i\in E_0$ for all $j<i<\ka$.
By construction, $\vec{e}$ is increasing, so
$\vec{V}=(U(b,e_i):i<\ka)$ is increasing in $\ro{p,X}$. 
Also by construction, $\vec{e}$ is unbounded in $E$ and
$(d,e_i)\trileq U(b,G(e_i))\trileq U(b,e_{i+1})$
for all $i<\ka$, so $\vec{V}$ is unbounded in $\ro{p,X}$.
Thus, $\vec{V}$ is a regular open escape sequence at $p$
with length $\add{E}$.
\end{proof}

\begin{remark}
The above proof works if we replace the regular open neighborhoods
of $p$ with an arbitrary complete lattice with its top removed.
\end{remark}

\begin{theorem}\label{escaperect}
If $\ka_0<\cdots<\ka_n$ are regular infinite cardinals, $X$ is $T_3$, 
$p\in X$, and $\nbhd{p,X}\eqcf\prod_{i\leq n}\ka_i$, then
$$
\escape{p,X}\cap\Reg=\escapero{p,X}\cap\Reg=\{\ka_0,\ldots,\ka_n\}.
$$
\end{theorem}
\begin{proof}
For each $i\leq n$, $\ka_i\in\escapero{p,X}$
by Lemma~\ref{escaperofilled} with 
$D=\prod_{j<i}\ka_j$ and $E=\prod_{i\leq j\leq n}\ka_j$.
(Note that $\prod\varnothing=\{\varnothing\}=1$.)
If $\lm$ is a regular infinite cardinal not equal to any $\ka_i$,
then $\lm\not\in\escape{p,X}$ by Lemma~\ref{escapegap} with 
$D=\prod(\{\ka_i:i<n\}\cap\lm)$ and
$E=\prod(\{\ka_i:i<n\}\setminus\lm)$.
\end{proof}

Theorem~\ref{t3rectcell} immediately follows from
Theorems~\ref{escapecell} and \ref{escaperect}.

\section{Skinny rectangles}\label{skinny}

To show that $\om\times\om_2$ cannot be cofinally equivalent 
to every neighborhood filter of a compactum, we use free sequences.
Recall that a transfinite sequence $\vec{p}$ in a space $X$ is
\emph{free} if $\closure{\{p_j:j<i\}}$ and $\closure{\{p_j:j\geq i\}}$
are disjoint for all $i$. Also, recall that $\hatfree{X}$ is
the least cardinal $\ka$ such that 
$X$ has no free sequence of length $\ka$.

\begin{lemma}\label{freeescape}
If $X$ is a compactum, $\alpha$ is a limit ordinal, and
$X$ has a free sequence of length $\alpha$, then
$X$ has an escape sequence of length $\alpha$ at some point.
\end{lemma}
\begin{proof}
Let $\vec{p}$ be a free sequence of length $\alpha$ in $X$.
Choose $q\in\bigcap_{i<\alpha}\closure{\{p_j:j\geq i\}}$.
Let $U_i=X\setminus\closure{\{p_j:j<i\}}$ for each $i<\alpha$,
so that $\vec{U}$ is an escape sequence at $q$.
\end{proof}

\begin{corollary}\label{freespectrum}
If $X$ is a compactum, then $\bigcup_{p\in X}\escape{p,X}$
includes every infinite cardinal less than $\hatfree{X}$.
\end{corollary}

\begin{lemma}\label{escapefree}
If $X$ is a compactum, $p\in X$,
and $\ka\in\escape{p,X}\cap\Reg$,
then $X$ has a free sequence of length $\ka$.
\end{lemma}
\begin{proof}
Let $(U_i:i<\ka)$ be an escape sequence at $p$.
If we replace each $U_i$ with its interior,
then $\vec{U}$ remains an escape sequence at $p$,
so \wma\ that each $U_i$ is open.
For each each $i<\ka$, choose a closed $V_i\in\nbhd{p,X}$
such that $V_i\subseteq U_i$.
For each $\sigma\in[\ka]^{<\om}$, there exists $i<\ka$ such that
$U_i\not\supseteq\bigcap_{j\in\sigma}V_j$. Since $\ka$
is regular, \wma\ that we have thinned out $\vec{U}$ 
such that $U_i\not\supseteq\bigcap_{j\in\sigma}V_j$ 
for all $i<\ka$ and $\sigma\in[i]^{<\om}$.
(Hence, $((V_i,U_i):i<\ka)$ is free sequence of regular
pairs in the sense of~\cite{todorfree}.)
By compactness, there exists $(x_i:i<\ka)$ such that
$x_i\in\bigcap_{j\leq i}V_j\setminus U_{i+1}$.
Moreover, $\vec{x}$ is free because $\{x_j:j<i\}$ is
contained in $X\setminus U_i$ and $\{x_j:j\geq i\}$ is
contained in $V_i$.
\end{proof}

\begin{theorem}\label{nongappoint}
If $X$ is a compactum and $\ka$ is a regular infinite cardinal, 
then $X$ has a neighborhood filter that is 
not cofinally equivalent to any 
$D\times E$ where $\cf(D)<\ka<\add{E}<\infty$.
\end{theorem}
\begin{proof}
Suppose that $p\in X$ and $\nbhd{p,X}$ is cofinally equivalent 
to some $D_p\times E_p$ where $\cf(D_p)<\ka<3Q2A4E\add{E_p}<\infty$.
By Lemma~\ref{escaperofilled}, $\add{E_p}\in\escapero{p,X}$;
by Lemma~\ref{escapefree}, $X$ has a free sequence of
length $\add{E_p}$, so $X$ has a free sequence of length $\ka$.
By Lemma~\ref{freeescape}, $X$ has an escape sequence 
of length $\ka$ at some point $q$. 
By Lemma~\ref{escapegap}, $\nbhd{q,X}$ is not
cofinally equivalent to any $D_q\times E_q$ where
$\cf(D_q)<\ka<\add{E_q}$.
\end{proof}

\begin{corollary}
Every compactum has a neighborhood filter that is 
not cofinally equivalent to $\om\times\om_2$.
\end{corollary}

The remainder of this section is devoted to proving
a stronger version of Corollary~\ref{freespectrum} for
homogeneous compacta in models of GCH.

\begin{definition}
Let the \emph{strict tightness} $\hattight{X}$ of a space $X$ 
be the least cardinal $\ka$ such that for every 
$A\subseteq X$  and $p\in\closure{A}$, we have 
$p\in\closure{B}$ for some $B\in[A]^{<\ka}$. 
The \emph{tightness} $\tight{X}$ of $X$ defined the same way,
except that we replace $[A]^{<\ka}$ with $[A]^{\leq\ka}$.
\end{definition}

Lemma~\ref{freetight} is due to \arhan\ and \shap\ 
(see~\cite{arhantightfree} or~\cite[Thm. 4.20]{monk})
for the case where $\ka$ is a successor cardinal (because the result
is stated in terms of tightness, not strict tightness).
However, as noted after Theorem 4.20 in~\cite{monk},
the proof clearly works for arbitrary regular infinite $\ka$.

\begin{lemma}\label{freetight}
If $X$ is a compactum and $\ka$ is a regular infinite cardinal,
then $\hattight{X}\leq\ka$ if and only if 
$X$ has no free sequence of length $\ka$.
\end{lemma}

\begin{definition} Given a space $X$, 
$\pi\character{X}=\sup_{p\in X}\pi\character{p,X}$.
\end{definition}

\begin{lemma}[\shap]\label{freepichar}
If $X$ is a compactum and $p\in X$, then
$X$ has a free sequence of length $\pi\character{p,X}$.
\end{lemma}
\begin{proof}
This lemma is just a localized form of \shap's Theorem,
$\pi\character{X}\leq\tight{X}$~\cite{shap}.
See the proof of a boolean algebraic version of \shap's Theorem 
in~\cite[Thm. 11.8]{monk}; it uses so\nbd-called free sequences 
of clopen sets and shows that our lemma holds if $X$ is 
zero\nbd-dimensional. To adapt that proof to the general case, 
simply replace free sequences of clopen sets with \todor's 
notion of free sequences of regular pairs~\cite{todorfree}.
\end{proof}

\begin{definition}
The \emph{weight} $\weight{X}$ of $X$ is 
the minimum of the cardinalities of bases of $X$.
\end{definition}

The next lemma is due to De La Vega~\cite[Thm. 3.2]{delavega},
except that we extend it to handle the case where
$\lm$ is weakly inaccessible.

\begin{lemma}\label{freecard}
If $X$ is a homogeneous compactum, $\lm$ is a regular infinite cardinal,
and $X$ has no free sequences of length $\lm$, 
then $\card{X}\leq 2^{<\lm}$.
\end{lemma}
\begin{proof}
First, the lemma is trivial when $\lm=\om$. Second,
by Lemma~\ref{freetight}, $\lm\geq\hattight{X}$. Therefore, if
$\lm$ is a successor cardinal, then Theorem 3.2 of~\cite{delavega},
which says $\card{X}\leq 2^{\tight{X}}$, implies $\card{X}\leq 2^{<\lm}$.
Finally, we can modify the proof of Theorem 3.2 of~\cite{delavega}
to show that $\card{X}\leq 2^{<\lm}$ without assuming $\lm$ is a
successor cardinal. 
In~\cite{delavega}, Theorem 3.2 is deduced from Theorem 3.1,
which assumes $\tight{X}\leq\ka$ and deduces $\weight{X}\leq 2^\ka$, 
where $\ka$ is an arbitrary infinite cardinal.
Thanks to regularity of $\lm$, we may safely respectively replace 
``$\leq\ka$'' and ``$2^\ka$'' with ``$<\lm$'' and ``$2^{<\lm}$'' 
throughout the statement and proof of Theorem 3.1. We also may safely
replace all sequences and sets of size $\ka$ with sequences and
sets of size less than $\lm$, and replace ``$\ka$\nbd-closed''
with ``$(<\lm)$\nbd-closed.'' These simple changes yield a proof of
$\weight{X}\leq 2^{<\lm}$. Therefore, it suffices to show that
$\card{X}\leq\weight{X}^{<\lm}$.

To deduce $\card{X}\leq 2^{\tight{X}}$ from $\weight{X}\leq 2^{\tight{X}}$,
De La Vega uses two inequalities,
$\card{X}\leq\weight{X}^{\pi\character{X}}$ and
$\pi\character{X}\leq\tight{X}$.
The first of these inequalities is due to Van Mill~\cite{vanmillcardphc} 
and it applies to all power homogeneous compacta.
(Using the same kinds of cosmetic changes as in the previous paragraph,
it is easy to check that his proof generalizes to
show that if $\pi\character{p,X}<\lm$ for all $p\in X$,
then $\card{X}\leq\weight{X}^{<\lm}$. However, 
since our $X$ is homogeneous, we do not need to make these changes.)
The second inequality localizes to Lemma~\ref{freepichar}, 
which implies $\pi\character{p,X}<\lm$ for all $p\in X$. 
Hence, $\card{X}\leq\weight{X}^{<\lm}$. 
\end{proof}

\begin{definition}\label{defchar}
Given a space $X$ and $A\subseteq X$,
\begin{itemize}
\item The \emph{character} $\character{A,X}$ of $A$ is the 
cofinality of the set of neighborhoods of $A$ (ordered by $\supseteq$);
\item the \emph{pseudocharacter} $\pseudocharacter{A,X}$ of $A$ is the
minimum of the cardinalities of families of neighborhoods of $A$
that have intersection $A$;
\item we abbreviate $\character{\{p\},X}$ by $\character{p,X}$
and $\pseudocharacter{\{p\},X}$ by $\pseudocharacter{p,X}$;
\item $\character{X}=\sup_{p\in X}\character{p,X}$.
\end{itemize}
\end{definition}
It is easily checked that $\character{A,X}=\pseudocharacter{A,X}$ 
whenever $A$ and $X$ are nested compacta.

The following lemma is due to \arhan~\cite{arhan}.
In~\cite{arhan}, it is stated in terms of tightness.
We state it as below for increased precision at limit cardinals.

\begin{lemma}\label{smallset}
If $\lm$ is an infinite cardinal, $X$ is a compactum, and
$X$ has no free sequence of length $\lm$, then
$X$ has a subset $A$ of size less than $\lm$ such that $\closure{A}$ 
has a nonempty closed subset $B$ such that $\pseudocharacter{B,X}<\lm$.
\end{lemma}

\begin{definition}\
\begin{itemize}
\item Given a space $X$, $\minchar{X}=\min_{p\in X}\character{p,X}$.
\item Given a cardinal $\lm$, $\log\lm=\min\{\ka:\lm\leq 2^\ka\}$.
\end{itemize}
\end{definition}

The next lemma is due to Juh\'asz~\cite{juhasz}.
Again, we state it differently 
for increased precision at limit cardinals.

\begin{lemma}\label{logminchar}
If $X$ is a compactum, $\lm$ is an infinite cardinal,
and $X$ has no free sequence of length $\lm$,
then $\log(\minchar{X})<\lm$.
\end{lemma}

\begin{corollary}\label{escapeminchar}
If $X$ is a compactum, $\lm$ is an infinite cardinal, and
$\lm$ is not in $\bigcup_{p\in X}\escape{p,X}$, then $\log(\minchar{X})<\lm$.
\end{corollary}

\begin{lemma}\label{charbound}
If $X$ is a space, $p\in X$, and $\mu\in\escape{p,X}\cap\Reg$, then
$\character{p,X}\geq\mu$.
\end{lemma}
\begin{proof}
If $P$ is a preorder and $\mu$ is a regular cardinal greater than $\cf(P)$,
then every map from $\mu$ to $P$ is bounded on a cofinal subset of
$\mu$, so there is no unbounded increasing $\mu$\nbd-sequence in $P$.
\end{proof}

\begin{theorem}[GCH]\label{gchhomog}
If $X$ is a homogeneous compactum and $p\in X$, then 
$\escape{p,X}$ is a closed initial segment 
of the infinite cardinals with maximum $\character{p,X}$.
\end{theorem}
\begin{proof}
We may assume that $\escape{p,X}$ is nonempty.
Let $\ka$ be the supremum of $\escape{p,X}$.
Let $\lm$ be the least infinite cardinal not in $\escape{p,X}$.
By Lemma~\ref{charbound}, $\character{p,X}$ is an upper bound of
$\escape{p,X}\cap\Reg$. Actually, $\character{p,X}$ bounds
$\escape{p,X}\setminus\Reg$ too: by \arhan's Theorem,
$\card{X}\leq 2^{\character{X}}$, which implies 
$\card{X}\leq 2^{\character{p,X}}$ by homogeneity; hence,
$\card{\nbhd{p,X}}\leq\character{p,X}^{++}$ by GCH.
Thus, $\ka\leq\character{p,X}$.
Therefore, it suffices to show that $\character{p,X}<\lm=\ka^+$.
By Lemmas~\ref{freeescape} and \ref{freecard}, 
$\card{X}\leq 2^{<\nu}$ where $\nu$ is the least regular cardinal $\geq\lm$.
By the \v Cech\nbd-Pospi\v sil Theorem, $\card{X}\geq 2^{\minchar{X}}$, 
which implies $\card{X}\geq 2^{\character{p,X}}$ by homogeneity. 
Therefore, by GCH, $\character{p,X}<\nu$, so $\character{p,X}\leq\lm$.
Since $\ka\leq\character{p,X}$, it follows that $\lm$ is 
the least infinite cardinal strictly above every $\mu\in\escape{p,X}$.
All that remains is to show that supremum of $\escape{p,X}$ is attained,
for if it is, then $\nu=\lm=\ka^+$. 
Seeking a contradiction, suppose that the supremum of $\escape{p,X}$ 
is not attained. We then have that $\ka$ is a limit cardinal and
$\ka=\lm$; by GCH, $\log(\ka)=\ka$.
By Corollary~\ref{escapeminchar}, $\log(\minchar{X})<\lm$;
by homogeneity, $\log(\character{p,X})<\lm$.
Therefore, $\log(\ka)\leq\log(\character{p,X})<\lm=\ka=\log(\ka)$,
which is absurd.
\end{proof}

De La Vega proved that GCH implies $\tight{X}=\character{X}$ for
all homogeneous compacta~\cite{delavega}. Letting $\free{X}$ denote 
the supremum of the cardinalities of free sequences in $X$, 
we have $\free{X}=\tight{X}$ for all compacta, by Lemma~\ref{freetight}.
Theorem~\ref{gchhomog} allows us to
deduce that the supremum $\free{X}$
is attained if GCH holds and $X$ is a homogeneous compactum.

\begin{corollary}[GCH]
If $X$ is a homogeneous compactum, then $\hatfree{X}=\character{X}^+$.
\end{corollary}
\begin{proof}
Fix $p\in X$. By Lemma~\ref{freeescape} and homogeneity,
$\hatfree{X}\leq\sup(\escape{p,X})^+$.
By Theorem~\ref{gchhomog}, $\max(\escape{p,X})=\character{p,X}$,
so $\hatfree{X}\leq\character{p,X}^+$. Moreover,
$\character{p,X}<\hatfree{X}$ by Lemma~\ref{escapefree}.
\end{proof}

It is easy to find inhomogeneous compacta $X$ where the supremum 
$\free{X}$ is not attained, $\free{X}<\character{X}$, or both.
For example, if $X$ is the one\nbd-point compactification 
of the topological sum $\oplus_{i<\al_{\al_{\om_1}}}Y_i$
where $Y_i$ is the ordered space $\al_j$ where
$\al_{\al_j}\leq i<\al_{\al_{j+1}}$, then
$\hatfree{X}=\free{X}=\al_{\om_1}$ and $\character{X}=\al_{\al_{\om_1}}$.

\section{Rectangles in $\betaoo$}\label{ultrafilters}

\begin{definition}\
\begin{itemize}
\item A preorder is \emph{cofinally rectangular} if it is
cofinally equivalent to a finite product of linear orders.
\item A preorder is \emph{cofinally scalene} if it
cofinally equivalent to some $\prod S$ where
$S\subseteq\Reg$.
\end{itemize}
\end{definition}

We now turn to the class of neighborhood filters
in the Stone-\v Cech remainder $\betaoo$, focusing on
the properties of being cofinally scalene, a natural
weakening of cofinally rectangular.
Since $\betaoo$ is not homogeneous~\cite{frolik}, 
we are leaving behind Van Douwen's Problem,
our initial motivation.
However, $\betaoo$ has been extensively studied,
so if we are to examine cofinally scalene 
neighborhood filters for their own sake, then 
$\betaoo$ is a reasonable place to start. 
(We will show that $\beta\om$ has no cofinally scalene
nonprincipal neighborhood filters.)

It is essentially shown in~\cite{isbell} that
$\nbhd{p,\betaoo}\eqcf([\cardc]^{<\om},\subseteq)$ for some 
$p\in\betaoo$. (Extend any independent family of size $\cardc$
to an ultrafilter avoiding pseudointersections of infinite 
subsets of the independent family.) The next (easy) theorem 
implies that $[\cardc]^{<\om}$ is not cofinally scalene.

\begin{theorem}\label{tukeyhigh}
If $\ka$ and $\lm$ are infinite cardinals, $\ka<\lm$,
and $\ka$ is regular, then
$([\lm]^{<\ka},\subseteq)$ is not cofinally scalene.
\end{theorem}
\begin{proof}
Seeking a contradiction, suppose $S\subseteq\Reg$ and
$[\lm]^{<\ka}\cup\prod S$ has a preordering $\trileq$ that
makes $[\lm]^{<\ka}$ and $\prod S$ cofinal suborders.
For each $i<\ka^+$, choose $f(i)\in\prod S$ such
that $\{i\}\trileq f(i)$. 
Since $\add{\prod S}=\add{[\lm]^{<\ka}}=\ka$, $\ka=\min(S)$.
Hence, there exist $\alpha<\ka$ and a cofinal
subset $I$ of $\ka^+$ such that $f(i)(\ka)=\alpha$ for
all $i\in I$. Since $\add{\prod(S\setminus\{\ka\})}>\ka$,
$f[J]$ is bounded for some $J\in[I]^{\ka}$. Hence,
$[J]^1$ is bounded in $[\lm]^{<\ka}$; this is our desired
contradiction.
\end{proof}

Isbell's Problem asks whether it is consistent with ZFC that all 
neighborhood filters of $\beta\om$ are cofinally equivalent to 
$[\cardc]^{<\om}$ or $1$.
In~\cite{milovichtuk}, it was shown that this is equivalent to asking
whether it is consistent with ZFC that all neighborhood filters of
$\betaoo$ are cofinally equivalent to $[\cardc]^{<\om}$.
Our next theorem solves an easier version of Isbell's Problem:
consistently, there are no cofinally scalene neighborhood
filters in $\betaoo$.

\begin{definition}\
\begin{itemize}
\item A point $p$ in a space $X$ is a \emph{P\nbd-point} if
$\nbhd{p,X}$ is $\om_1$\nbd-directed.
\item A filter $\mcF$ on $\om$ is assumed to be ordered by $\supseteq$, but
$\mcF_*$ denotes $\mcF$ ordered by eventual containment $\supseteq^*$.
\item $\beta\om$ is identified with the space of ultrafilters on $\om$.
\end{itemize}
\end{definition}
Note that if $\mcU\in\betaoo$, then $\mcU\eqcf\nbhd{\mcU,\beta\om}$
and $\mcU_*\eqcf\nbhd{\mcU,\betaoo}$.

\begin{theorem}
It is consistent with ZFC that no neighborhood filter
in $\betaoo$ is cofinally scalene.
\end{theorem}
\begin{proof}
Suppose that $S\subseteq\Reg$ and $\nbhd{p,\betaoo}\eqcf \prod S$.
There is a model of ZFC without P\nbd-points in 
$\betaoo$~\cite{shelahnoppoint}, so it suffices to show that
$\nbhd{p,\betaoo}$ is $\oml$\nbd-directed.
First, $\betaoo$ has no isolated points, so $\nbhd{p,\betaoo}\not\eqcf 1$.
Hence, $S\not=\varnothing$ and $\min(S)\geq\om$.
Second, by~\cite[Thm. 3.13]{milovichtuk}, 
we cannot have $\nbhd{p,\betaoo}\eqcf\om\times D$ where
$D$ is $\om_1$\nbd-directed. Since $\prod(S\setminus\{\om\})$
is $\oml$\nbd-directed (even if $S\setminus\{\om\}=\varnothing$), 
we cannot have $\om\in S$. Therefore, $\min(S)\geq\oml$, so 
$\prod S$ is $\oml$\nbd-directed. Therefore, 
$\nbhd{p,\betaoo}$ is also $\oml$\nbd-directed.
\end{proof}

\begin{remark}
For an alternative proof of the above theorem, 
force with finite binary partial functions on $\ka$
where $\om_1<\ka=\ka^\om$.
This yields a model of ZFC with P\nbd-points $\mcV$ in $\betaoo$
because $\cardd=\cardc$ (see~\cite[Thm. 9.25]{blasshandbook}), 
but they all satisfy $\mcV_*\eqcf[\cardc]^{<\om_1}$ because the
generic sequence of Cohen reals has length $\cardc$ and none of its
uncountable subsequences has an infinite pseudointersection.
In this model, $\cardc=\ka$, so $[\cardc]^{<\om_1}$ is not
cofinally scalene by Theorem~\ref{tukeyhigh}.
\end{remark}

On the other hand, it is well known that MA($\sigma$\nbd-centered)
implies that $\cardc$ is regular and $\nbhd{p,\betaoo}\eqcf\cardc$
for some $p\in\betaoo$. (See~\cite[Thms. 7.12, 7.14]{blasshandbook}.) 
In~\cite{milovichtuk}, it was shown that MA($\sigma$\nbd-centered) 
also implies that for every regular infinite cardinal $\ka\leq\cardc$, 
$\nbhd{p,\betaoo}\eqcf[\cardc]^{<\ka}$ for some $p\in\betaoo$.
Are these all the cofinal types of neighborhood filters in $\betaoo$
implied to exist by MA($\sigma$\nbd-centered)? Does
MA($\sigma$\nbd-centered) imply that every cofinally scalene
neighborhood filter in $\betaoo$ is cofinally equivalent to $\cardc$?
The rest of this section develops some partial answers to these questions.

\begin{definition}\label{deftukey}\
\begin{itemize}
\item A map between directed sets is \emph{Tukey}
if it sends unbounded sets to unbounded sets.
\item A map between directed sets is \emph{convergent}
if it sends cofinal sets to cofinal sets.
\item $D\leq_T E$ means there is a Tukey map $f\colon D\rightarrow E$.
\item $E\geq_T D$ means there a convergent map $g\colon E\rightarrow D$.
\end{itemize}
\end{definition}
It is easy to check that $D\leq_T E$ if and only if $E\geq_T D$. 
Tukey introduced the relation $\geq_T$ in~\cite{tukey} and there
proved that $D\geq_T E\geq_T D$ if and only if $D\eqcf E$.
(Tukey originally, equivalently defined $E\geq_T D$ to mean that
there exist maps $f\colon D\rightarrow E$ 
and $g\colon E\rightarrow D$ such that 
$e\geq f(d)\Rightarrow g(e)\geq d$.)
The following lemma is implicit in~\cite{tukey}.

\begin{lemma}\label{tukeyprod}
Given directed sets $A$, $B$, and $C$,
we have $A\times B\leq_T C$ if and only if 
$A\leq_T C$ and $B\leq_T C$.
In particular, if $A\leq_T B$, 
then $B\eqcf A\times B$.
\end{lemma}
\begin{proof}
First, if $h\colon A\times B\rightarrow C$ is
Tukey, then, for any fixed $(a_0,b_0)\in A\times B$, 
the maps $f(\bullet,b_0)\colon A\rightarrow C$ 
and $g(a_0,\bullet)\colon B\rightarrow C$ are also Tukey.
Second, if $p\colon A\rightarrow C$ and $q\colon B\rightarrow C$
are Tukey, then any $r\colon A\times B\rightarrow C$ satisfying
$p(a),q(b)\leq r(a,b)$ is also Tukey.
Third, $B\leq_T A\times B$ is witnessed by $b\mapsto(a_0,b)$
for any fixed $a_0\in A$.
Hence, if $A\leq_T B$, then $A\times B\leq_T B\leq_T A\times B$,
which implies $B\eqcf A\times B$
\end{proof}

\begin{definition}\
\begin{itemize}
\item Fix a pairing function $\la\bullet,\bullet\ra$ 
from $\om\times\om$ to $\om$.
\item For all $E\subseteq\om$ and $i\in\om$, $(E)_i=\{j:\la i,j\ra\in E\}$.
\item Given filters $\mcF,\mcG$ on $\om$, the \emph{Fubini product}
$\mcF\otimes\mcG$ is 
$$
\{E\subseteq\om:\{i:(E)_i\in\mcG\}\in\mcF\}.
$$
\end{itemize}
\end{definition}

\begin{lemma}[CH]\label{betaooreduction}
Every nonprincipal neighborhood filter in $\beta\om$ is 
cofinally equivalent to a neighborhood filter in $\betaoo$.
\end{lemma}
\begin{proof}
Let $\mcU\in\betaoo$ and let $\mcV$ be P\nbd-point in $\betaoo$.
It suffices to show that $\mcU\eqcf(\mcU\otimes\mcV)_*$.
The map $E\mapsto\{\la i,j\ra: (i,j)\in E\times\om\}$ is
Tukey from $\mcU$ to $(\mcU\otimes\mcV)_*$, so it suffices to
show that $(\mcU\otimes\mcV)_*\leq_T\mcU$. Every nonprincipal 
ultrafilter on $\om$ has uncountable cofinality. Hence, we can use 
CH to get a bijection $h$ from $\om_1$ to a cofinal subset of $\mcU$
such that $h$ is nondecreasing, \ie\ 
$h(\beta)\not\supseteq h(\alpha)$ for all $\alpha<\beta<\oml$.
The map $h$ is necessarily Tukey, so $\om_1\times\mcU\eqcf\mcU$
by Lemma~\ref{tukeyprod}. Therefore, it suffices to show that
$(\mcU\otimes\mcV)_*\leq_T\mcU\times\om_1$.

Define $\pi\colon(\mcU\otimes\mcV)_*\rightarrow\mcU$ by 
$\pi(E)=\{i:(E)_i\in\mcV\}$. Let 
$\psi\colon(\mcU\otimes\mcV)_*\rightarrow\om_1$ be an arbitrary 
injection. It suffices to show that $(\pi,\psi)$ is Tukey.
Since $\psi$ sends uncountable sets to unbounded sets,
it suffices to show that $\pi$ sends countable unbounded sets to
unbounded sets. We will prove the contrapositive. Suppose that
$A_n\in\mcU\otimes\mcV$ for all $n<\om$ and $\{\pi(A_n):n<\om\}$ is
bounded in $\mcU$. We need to show that $\mcA=\{A_n:n<\om\}$ is
bounded in $(\mcU\otimes\mcV)_*$. Since $\pi[\mcA]$ is bounded, 
$B=\bigcap\pi[\mcA]\in\mcU$. Choose $C\subseteq\om$ such
that $(C)_i=\varnothing$ for all $i\not\in B$ and, for all $j\in B$,
$(C)_j\in\mcV$, $(C)_j\subseteq\bigcap_{n\leq j}(A_n)j$, and
$(C)_j\subseteq^* (A_n)_j$ for all $n<\om$. We can find such $C$ 
because $\mcV$ is a P\nbd-point and $(A_n)_j\in\mcV$ for all 
$(n,j)\in\om\times B$. It follows that $C\in\mcU\otimes\mcV$
and $C\subseteq^* A_n$ for all $n<\om$, so $\mcA$ is bounded in 
$(\mcU\otimes\mcV)_*$ as desired.
\end{proof}

By Theorem 44 of~\cite{dobtor}, if $\cardd=\cardu=\cardc$, then
there are $2^\cardc$\nbd-many cofinal types of neighborhood filters 
in $\beta\om$. Our next theorem shows that exactly one of these
$2^{\cardc}$ cofinal types is cofinally scalene and, assuming CH,
transfers this result to $\betaoo$.

\begin{theorem}[CH]
If $X$ is $\betaoo$ or $\beta\om$, then there are $2^\cardc$ 
cofinal types of neighborhood filters in $X$, but all the cofinally 
scalene neighborhood filters in $X$ are cofinally equivalent to 
$\cardc$ if $X=\betaoo$; $1$ if $X=\beta\om$.
\end{theorem}
\begin{proof}
By Lemma~\ref{betaooreduction}, the
$2^\cardc$ cofinal types of neighborhood filters of $\beta\om$
are also instantiated by neighborhood filters in $\betaoo$.
To prove the second half of the theorem, fix $\mcU\in\betaoo$. 
By Theorems~3.13 and 3.16 of~\cite{milovichtuk},
$\mcU_*\not\eqcf\om\times D$ and $\mcU\not\eqcf\om\times D$
if $\om<\add{D}$. Therefore, 
if $\mcU$ is not a P\nbd-point in $\betaoo$, then $\mcU_*$
is not cofinally scalene. If $\mcU$ is a P-point in $\betaoo$, 
then $\mcU_*\eqcf\cardc$ by CH. Finally, $\mcU$ is not a
P\nbd-point in $\beta\om$ because it is in the closure of $\om$,
so $\mcU$ is not cofinally scalene, but every principal
ultrafilter is.
\end{proof}

\begin{remark}
The proof that all cofinally scalene neighborhood filters in 
$\beta\om$ are principal did not use CH.
\end{remark}

For cofinally scalene neighborhood filters in $\betaoo$, CH
is a relatively uninteresting context because the only possible
cofinality is $\om_1$, so the only possible cofinal types 
are $\om_1$ and $\om\times\om_1$. Our next theorem's
hypotheses allow $\cardc$ to be arbitrarily large.

\begin{definition}
Given a class $\Gamma$ of forcings, let $\cardm_\Gamma$ denote
the least cardinal $\ka$ such that there is a forcing $\mbP\in\Gamma$
and a family $\mcD$ of $\ka$\nbd-many dense subsets of $\mbP$
such that no filter of $\mbP$ meets every $D\in\mcD$.
\end{definition}

\begin{lemma}\label{msnlinked}
If $1\leq n<\om$, $\ka$ is a regular infinite cardinal, 
$\ka<\cardml{n}$, and $\ka<\add{D}$,
then no neighborhood filter in $\betaoo$ is 
cofinally equivalent to $\ka\times D$.
\end{lemma}
\begin{proof}
The proof goes like the author's proof of
Theorem~3.13 of~\cite{milovichtuk}, which handled
the case where $n=1$ and $\ka=\om$.
(Note that $\om_1=\cardml{1}$ because all forcings
are $1$\nbd-linked.)
Fix $\mcU\in\betaoo$. Seeking a contradiction, suppose
that $\trileq$ is a preordering of $\mcU\cup(\ka\times D)$
that makes $\mcU_*$ and $\ka\times D$ cofinal suborders.
We then have $\add{\mcU_*}=\ka$.
Therefore, fixing $a\in D$, we can find $(F(i,a):i<\ka)$
such that $(i,a)\trileq F(i,a)\in\mcU$ and
$F(j,a)\supseteq^* F(i,a)$ for all $j<i<\ka$.
Next, for each $d\in D\setminus\{a\}$, choose 
$(F(i,d):i<\ka)$ such that $(i,d)\trileq F(i,d)\in\mcU$ and
$F(j,a)\supseteq^* F(i,d)$ for all $j\leq i<\ka$.

Inductively construct $g\colon\ka\rightarrow\ka$ as follows.  
Given $i<\ka$ and $g\restrict i$, if $d\in D$, then 
$\{F(g(j),d):j<i\}\trileq(l,b)$ for some $l<\ka$ and $b\in D$.  
Since $\ka<\add{D}$, there exist $l<\ka$ and a cofinal subset 
$C_i$ of $D$ such that for all $c\in C_i$ there exists $b\in D$ 
such that $\{F(g(j),c):j<i\}\trileq(l,b)$.
Choose $g(i)\geq l$ such that $g(j)<g(i)$ for all $j<i$.
This completes the construction of $g$. Observe that $g$ is 
strictly increasing and therefore has range cofinal in $\ka$.

Set $m=n+1$. 
For each $j<m$, let $I_j$ be the ideal generated by
$\{F(g(2mi+2j),a)\setminus F(g(2mi+2j+2),a):i<\ka\}$.
Observe that $X\cap Y$ is finite for all $X\in I_s$
and $Y\in I_t$ where $s<t<m$.
Since $\ka<\cardml{n}$, Corollary 21 of~\cite{aviles} implies
that there exist $A_0,\ldots,A_{m-1}\subseteq\om$ such that 
$\bigcap_{j<m}A_j=\varnothing$ and, for all $j<m$ and
$X\in I_j$, $X\subseteq^*A_j$.
Choose $j<m$ such that $A_j\not\in\mcU$;
choose $\alpha<\ka$ and $d\in D$ such that 
$\omega\setminus A_j\trileq(\alpha,d)$.
Choose $i<\ka$ such that $\alpha\leq g(2mi+2j)$;
choose $c\in C_{2mi+2j+1}$ such that $d\leq c$.
Since $\omega\setminus A_j\trileq(g(2mi+2j),c)$, we have
$\omega\setminus A_j\supseteq^*F(g(2mi+2j),c)$.
Since $F(g(2mi+2j),a)\supseteq^*F(g(2mi+2j),c)$, we have
$F(g(2mi+2j),a)\setminus A_j\supseteq^* F(g(2mi+2j),c)$.
Hence, $F(g(2mi+2j+2),a)\supseteq^* F(g(2mi+2j),c)$,
which implies $(2mi+2j+2,a)\trileq F(g(2mi+2j),c)$.
By our choice of $c$, $F(g(2mi+2j),c)\trileq(2mi+2j+1,b)$
for some $b\in D$. Therefore, $(2mi+2j+2,a)\trileq(2mi+2j+1,b)$.
We now have our desired contradiction:
$(2mi+2j+2,a)\leq(2mi+2j+1,b)$ but $2mi+2j+2\not\leq 2mi+2j+1$.
\end{proof}

\begin{remark}
By Theorem~24 of~\cite{aviles} and the fact that
$\cardmcen$ is regular (see~\cite[Thms. 7.12, 7.14]{blasshandbook}), 
it is consistent with ZFC that
$$
\cardml{n}<\sup_{1\leq k<\om}\cardml{k}<\cardmcen
$$ 
for all $n\in[1,\om)$. Also,
the proof of Theorem~5.10 of~\cite{milovichnts} shows that, 
given any pair $(\ka,\lm)$ of uncountable regular cardinals 
satisfying $\lm=\lm^{<\ka}$, some $\mbP$ satisfying Knaster's
condition forces $\cardmctbl=\cardmcen=\ka$, $\cardc=\lm$, 
and the existence of a neighborhood filter of $\betaoo$ 
that is cofinally equivalent to $\ka\times\lm$.
\end{remark}

\begin{theorem}\label{msnlinkedrect}
If MA($\sigma$\nbd-$n$\nbd-linked) holds for some $n\in[1,\om)$,
then every cofinally scalene neighborhood filter 
in $\betaoo$ is cofinally equivalent to $\cardc$.
\end{theorem}
\begin{proof}
Fix $\mcU\in\betaoo$.
It is well known that MA($\sigma$\nbd-$n$\nbd-linked)
implies MA(countable) implies $\cf(\mcU_*)=\cardc$
(see~\cite[Thms. 7.13, 5.19, 9.7]{blasshandbook}),
so if $\mcU_*$ is cofinally scalene, then
$\mcU_*\eqcf\prod S$ where 
$\varnothing\not=S\subseteq\Reg\cap[\om,\cardc]$.
By Lemma~\ref{msnlinked}, $\min(S)\geq\cardml{n}$,
so $S=\{\cardc\}$.
\end{proof}

\begin{remark}
MA($\sigma$\nbd-$1$\nbd-linked) is equivalent to CH.
\end{remark}

\begin{question}
Can the hypothesis of the above theorem can be weakened
to MA($\sigma$\nbd-centered)? MA(countable)? 
\end{question}

\section{Products and Fubini products}\label{fubini}

The proof of Lemma~\ref{betaooreduction} used
Fubini products to partially answer questions about
the frequency of cofinally scalene posets of the
form $\mcU_*$ where $\mcU\in\betaoo$: CH implies that there 
are $2^\cardc$ cofinal types of the form $\mcU_*/\eqcf$, but
only one is cofinally scalene. Working in
the opposite direction, we now use product orders 
to answer a question about the cofinal types of Fubini products. 
Given a filter $\mcF$ on $\om$, 
adopt the notation $\mcF^{\otimes 1}=\mcF$ and 
$\mcF^{\otimes n+1}=\mcF\otimes\mcF^{\otimes n}$,
noting that the Fubini product is associative modulo 
the order isomorphisms induced by 
$\la\la i,j\ra, k\ra\leftrightarrow\la i,\la j,k\ra\ra$.
Question~39 of~\cite{dobtor} asks whether there is an 
ultrafilter $\mcU$ on $\omega$ such that 
$\mcU<_T\mcU^{\otimes 2}<_T\mcU^{\otimes 3}<_T[\cardc]^{<\om}$.
Our answer is a strong ``no.'' 

\begin{lemma}\label{fubinidirect}
If $\mcF$ and $\mcG$ are nonprincipal filters on $\om$,
then $\mcF\otimes\mcG\eqcf\mcF\times\mcG^\om$.
\end{lemma}
\begin{proof}
By Lemma~\ref{tukeyprod}, it suffices to show that
(1) $\mcF\leq_T\mcF\otimes\mcG$, (2) $\mcG^\om\leq_T\mcF\otimes\mcG$,
and (3) $\mcF\otimes\mcG\leq_T\mcF\times\mcG^\om$.
First, the map $E\mapsto\{\la i,j\ra: (i,j)\in E\times\om\}$ is
Tukey from $\mcF$ to $\mcF\otimes\mcG$.
Second, let us construct a convergent map 
$\Phi\colon\mcF\otimes\mcG\rightarrow\mcG^\om$.
Given $A\in\mcF\otimes\mcG$ and $i<\om$, 
set $\pi(A)=\{j: (A)_j\in\mcG\}$ and
$\Phi(A)(i)=(A)_k$ where $k=\min(\pi(A)\setminus i)$.
Suppose that $\mcC\subseteq\mcF\otimes\mcG$ is cofinal and $\xi\in\mcG^\om$.
To prove that $\Phi$ is convergent, it suffices to show that 
$\Phi(A)\geq\xi$ for some $A\in\mcC$.
Set $\zeta(i)=\bigcap_{j\leq i}\xi(j)$ for all $i<\om$.
Choose $A\in\mcC$ such that $A\subseteq\bigcup_{i<\om}(\{i\}\times\zeta(i))$.
Then, for all $i<\om$, 
$$
\Phi(A)(i)\subseteq\zeta(\min(\pi(A)\setminus i))
\subseteq \zeta(i)\subseteq \xi(i),
$$
so $\Phi(A)\geq\xi$ as desired.

Finally, we construct a convergent map 
$\Psi\colon\mcF\times\mcG^\om\rightarrow\mcF\otimes\mcG$.
Given $(A,\xi)\in\mcF\times\mcG^\om$, 
set $\Psi(A,\xi)=\bigcup_{i\in A}(\{i\}\times \xi(i))$.
The map $\Psi$ is surjective and order preserving, so it is convergent.
\end{proof}

\begin{theorem}\label{squarecube}
For all nonprincipal filters $\mcF,\mcG$ on $\om$,
$\mcF\otimes\mcG\eqcf\mcF\otimes \mcG^{\otimes 2}$.
In particular, $\mcF^{\otimes 2}\eqcf\mcF^{\otimes 3}$.
\end{theorem}
\begin{proof} Use Lemma~\ref{fubinidirect} three times.
\begin{align*}
\mcF\otimes(\mcG\otimes\mcG)&\eqcf\mcF\times(\mcG\otimes\mcG)^\om\\
                            &\eqcf\mcF\times(\mcG\times\mcG^\om)^\om\\
                            &\eqcf\mcF\times\mcG^\om\\
                            &\eqcf\mcF\otimes\mcG.
\end{align*}
\end{proof}

\begin{remark}
The second half of the above theorem, $\mcF^{\otimes 2}\eqcf\mcF^{\otimes 3}$, 
restricted to the case where $\mcF$ is an ultrafilter, 
was first proven by Andreas Blass using nonstandard models of arithmetic. 
After learning of the result from Blass, the author, 
having already proved Lemma~\ref{fubinidirect},
immediately thought of the above short proof, 
only later reading Blass' less direct proof (which is unpublished).
\end{remark}

As another application of Lemma~\ref{fubinidirect}, we show that
the Fubini product is commutative modulo cofinal equivalence among the
nonprincipal P\nbd-filters on $\om$ (\ie\ nonprincipal filters $\mcF$ 
on $\omega$ such that if $\vec{A}\in\mcF^\om$,
then, for some $B\in\mcF$, $B\subseteq^* A_i$ for all $i<\om$).

\begin{theorem}\label{fubinicommute}
If $\mcF$ and $\mcG$ are nonprincipal filters on $\om$
and $\mcG$ is a P-filter, then $\mcF\otimes\mcG\eqcf\mcF\times\mcG\times\om^\om$.
Therefore, if $\mcF$ is also a P-filter, then $\mcF\otimes\mcG\eqcf\mcG\otimes\mcF$.
\end{theorem}
\begin{proof}

By Lemma~\ref{fubinidirect}, it suffices to show that
$\mcG^\om\eqcf\mcG\times\om^\om$. By Lemma~\ref{tukeyprod},
it suffices to show that (1) $\mcG\leq_T\mcG^\om$, 
(2) $\om^\om\leq_T \mcG^\om$,
and (3) $\mcG^\om\leq_T\mcG\times\om^\om$.
First, the diagonal map from $\mcG$ to $\mcG^\om$ is Tukey.
Second, $(n_i:i<\om)\mapsto(\om\setminus n_i:i<\om)$ is a Tukey map
from $\om^\om$ to $\mcG^\om$.
(Alternatively, (2) follows from Fact~31 and Theorem~32 of~\cite{dobtor},
which are stated for ultrafilters but 
have proofs that work for filters in general.)
Finally, following the proof of Theorem~33 in~\cite{dobtor},
map each $\vec{A}\in\mcG^\om$ to some $(B,h)\in\mcG\times\om^\om$
such that $B\setminus h(i)\subseteq A_i$ for all $i<\om$.
Again, the map is Tukey.
\end{proof}

\begin{remark}
Under the additional hypothesis that $\mcG\geq_T\om^\om$, 
Theorem~\ref{fubinicommute} combines with Lemma~\ref{tukeyprod} to
conclude $\mcF\otimes\mcG\eqcf\mcF\times\mcG$. Therefore,
Theorem~\ref{fubinicommute} improves upon Corollary~34 of~\cite{dobtor},
which says that if $\mcF$ and $\mcG$ are nonprincipal filters on $\om$,
$\mcG$ is a P-filter, and $\mcG\geq_T\om^\om$, 
then $\mcF\otimes\mcG\eqcf\mcF\times\mcG$.
(Technically, Corollary~34 of~\cite{dobtor} is stated only for
ultrafilters, but its proof works for filters in general.)
\end{remark}

All the results of this section naturally generalize to the
$\ka$\nbd-complete uniform filters on an arbitrary infinite regular $\ka$.

\end{document}